\documentclass[10pt]{article}

\usepackage[margin=1in]{geometry} % set the margins to 1in on all sides
\usepackage{graphicx} % to include figures
\usepackage{amsmath} % great math stuff
\usepackage{amsfonts} % for blackboard bold, etc
\usepackage{amsthm} % better theorem environments
\usepackage{amssymb} %for \nsubseteqq
\usepackage{enumerate}
\usepackage{algorithm}
\usepackage[noend]{algpseudocode}
\usepackage{authblk}
\usepackage{eucal}
\usepackage{url}
\usepackage{enumitem}
\usepackage{hyperref}
\usepackage{csvsimple}
\usepackage{caption}
\usepackage{tikz}
\usepackage{sectsty}
\usepackage{verbatim}

\newtheorem{thm}{Theorem}[section]
\newtheorem{lem}[thm]{Lemma}

\newtheorem{defn}[thm]{Definition}

\newtheorem{constr}[thm]{Construction}

\newcommand{\NN}{\mathbb{N}}%for natural numbers

\newcommand{\RNum}[1]{({\uppercase\expandafter{\romannumeral #1\relax}})}

\newcommand{\tr}[1]{{\rm trace \,}}
\newcommand{\keywords}[1]{\textbf{\textit{Keywords:}}#1}

\sectionfont{\fontsize{12}{15}\selectfont}

\begin{document}
\title{A New Method for Constructing Circuit Codes}
\author{Kevin M. Byrnes
\thanks{E-mail:\texttt{dr.kevin.byrnes@gmail.com}}}
\maketitle

\begin{abstract}
Circuit codes are constructed from induced cycles in the graph of the $n$ dimensional hypercube.  They are both theoretically and practically important, as circuit codes can be used as error correcting codes.  When constructing circuit codes, the length of the cycle determines its accuracy and a parameter called the spread determines how many errors it can detect.  We present a new method for constructing a circuit code of spread $k+1$ from a circuit code of spread $k$.  This method leads to record code lengths for circuit codes of spread $k=7 \text{ and } 8$ in dimension $22\le n\le 30$.  We also derive a new lower bound on the length of circuit codes of spread 4, improving upon the current bound for dimension $n\ge 86$.

%We present a method for constructing a circuit code of spread $k+1$ from a circuit code of spread $k$.  We report new lower bounds on the lengths of circuit codes for $k=7,8$ and dimension $22\le n \le 30$ arising from our construction (in particular for spread $k=7$ in dimension $n=22$), and present the most comprehensive and up to date table of circuit code lower bounds of which we are aware.  We also derive a new lower bound on the length of circuit codes of spread 4, improving upon the current bound for dimension $n\ge 86$.
\end{abstract}


\section{Introduction}
Let $I(n)$ denote the graph of the $n$ dimensional hypercube, that is the graph on $2^n$ vertices where each vertex corresponds to a binary vector of length $n$, and two vertices $x$ and $x'$ are adjacent if and only if their binary vectors differ in exactly one position.  For any vertex induced subgraph $G$ of $I(n)$ and any two vertices $x,x'\in G$ we define the distance $d_G(x,x')$ as the minimum number of edges in $G$ needed to travel from $x$ to $x'$.  If there is no path in $G$ from $x$ to $x'$ then $d_G(x,x')=\infty$.  Observe that $d_{I(n)}(x,x')$ equals the number of positions where the binary vectors corresponding to $x$ and $x'$ differ.

\begin{defn}
An induced subgraph $C$ of $I(n)$ is a circuit code of spread $k$ (an $(n,k)$ circuit code) if:
\begin{enumerate}
\item $C$ is a circuit (i.e. an undirected cycle).
\item If $x$ and $x'$ are vertices of $C$ with $d_{I(n)}(x,x')<k$ then $d_C(x,x')=d_{I(n)}(x,x')$.
\end{enumerate}
\end{defn}

An equivalent characterization of circuit codes was proven by Klee.
\begin{lem}[Klee \cite{Klee} Lemma 2]
\label{Fact2}
An $n$-dimensional circuit code $C$ of length $N\ge 2k$ has spread $k$ if and only if for all vertices $x,x'\in C$, ${d_C(x,x')\ge k \Rightarrow d_{I(n)}(x,x')\ge k}$.
\end{lem}

Finding long circuit codes is practically and theoretically important, since circuit codes can be used as error-correcting codes\cite{Kautz}.  Circuit codes of spread 1 are known as Gray codes\cite{Gray}, and circuit codes of spread 2 are known as snakes (or coils)\cite{Kautz}, both of these have been extensively studied.  Let $K(n,k)$ denote the maximum length of an $(n,k)$ circuit code, it is well-known that $K(n,1) = 2^n$ and $K(n,2)\ge \frac{77}{256}2^n$\cite{AK}.  In contrast, circuit codes of spread $k\ge 3$ are less-well understood and exact values for $K(n,k)$ are generally only known for $n\le 17$ and $k\le 7$ and some special $(n,k)$ pairs.
%In contrast, circuit codes of spread $k\ge 3$ are less-well understood and exact values for $K(n,k)$ are only known for some special $(n,k)$ pairs.  Hence developing constructions leading to new theoretical and empirical lower bounds for $K(n,k)$ has been an active area of research.

In this note we present a simple new construction for generating a circuit code of spread $k+1$ from a circuit code of spread $k$.  This allows the better studied codes of smaller spreads to be leveraged to create codes of larger spreads, and results in several new records for codes of spread 7 and 8, and dimension $22\le n\le 30$.  Specifically, we prove the following theorem.

\begin{thm}
\label{Thm1}
Let $C$ be an $(n,k)$ circuit code with length $N\ge 2(k+1)$.  Then there exists an $(n+r,k+1)$ circuit code $C'$ with length $N'=N+2\lceil \frac{N}{2(k+1)} \rceil$, where $r=\lceil \log_2 \frac{N}{2(k+1)}\rceil+1$.
\end{thm}

A useful consequence of Theorem \ref{Thm1} is a new lower bound on $K(n,4)$ which improves upon the current lower bound when $n\ge 86$.
\begin{thm}
\label{Thm2}
For $n\ge 6, K(\lfloor 1.53 n\rfloor,4)\ge 40\cdot 3^{(n-8)/3}$, and hence $K(n,4)\ge 40\cdot 3^{(\lfloor .6535n \rfloor-8)/3}$.
\end{thm}

\section{Previous Constructions and Bounds}
We begin by surveying the theoretical lower bounds for $K(n,k)$ and some of the most important constructions used in their derivation.  Exact values for $K(n,k)$ are known for only a few special cases, given in Table \ref{ExactBounds}.
\begin{table}[H]
\caption{Exact values for $K(n,k)$.}
\label{ExactBounds}
\centering
\begin{tabular} {l l l}
$K(n,k)=2n$ & for $n<\lfloor \frac{3k}{2} \rfloor+2$ & (See \cite{Singleton})\\
$K(\lfloor \frac{3k}{2} \rfloor +2,k)=4k+6$ & for $k$ even & (See \cite{Douglas})\\
$K(\lfloor \frac{3k}{2} \rfloor +2,k)=4k+4$ & for $k$ odd & (See \cite{Douglas})\\
$K(\lfloor \frac{3k}{2} \rfloor +3,k)=4k+8$ & for $k$ odd $\ge 9$ & (See \cite{Douglas})\\
\end{tabular}
\end{table}

The following constructions apply for a wide variety of $(n,k)$ combinations.  %They serve as the basis for the theoretical lower bounds on $K(n,k)$ given in Table \ref{LowerBounds} and are  frequently used in combination with computational techniques for computing circuit codes to update the lower bounds for multiple $(n,k)$ combinations once a new record has been established for a particular $(n,k)$ pair.%  
Here we state the ``result'' of each construction and refer the reader to the original paper for the precise construction details.  

\begin{constr}[Singleton \cite{Singleton}]
\label{Construction1}
Let $C$ be an $(n,k)$ circuit code with length $N$.  Then there exists an $(n+1,k)$ circuit code $C'$ with length $N'=N+2\lfloor \frac{N}{2k}\rfloor$.
\end{constr}

\begin{constr}[Singleton \cite{Singleton}]
\label{Construction2}
Let $C$ be an $(n,k)$ circuit code with length $N$, and $k\ge 3$.  Then there exists an $(n+2,k)$ circuit code $C'$ with length $N'=N+4\lfloor \frac{N}{2(k-1)}\rfloor$.
\end{constr}

\begin{constr}[Singleton \cite{Singleton}]
\label{Construction3}
Let $C$ be an $(n,k)$ circuit code with length $N$ for $k\ge 3$ and $k$ odd.  Then there exists an $(n+\frac{k+1}{2},k)$ circuit code $C'$ with length $N'=N+(k+1)\lfloor\frac{N}{k+1}\rfloor$.
\end{constr}

\begin{constr}[Singleton \cite{Singleton}]
\label{Construction4}
Let $C$ be an $(n,k)$ circuit code of length $N$ with $k\ge2$ and $k$ even.  Then there exists an $(n+\frac{k+2}{2},k)$ circuit code $C'$ of length $N'=N+(k+2)\lfloor \frac{N}{k+1}\rfloor$. 
\end{constr}

\begin{constr}[Deimer \cite{Deimer}]
\label{Construction5}
Let $C$ be an $(n+1,k+1)$ circuit code of length $N$.  Then there exists an $(n,k)$ circuit code $C'$ of length $N'\ge N-\lfloor \frac{N}{n+1} \rfloor$.
\end{constr}

\begin{constr}[Klee \cite{Klee}]
\label{Construction6}
Let $k$ be even and let $2\le n_1\le n_2$.  Suppose $C^1$ is an $(n_1,k-1)$ circuit code of length $N^1\ge 2k$ where $N^1$ is divisible by $k$, and suppose $C^2$ is an $(n_2,k)$ circuit code with length $N^2\ge 2k$.  If $k=2$ there exists an $(n_1+n_2,k)$ circuit code $C^3$ of length $N^3=\frac{N^1N^2}{k}$.  If $k\ge 4$ there exists an $(n_1+n_2+1,k)$ circuit code $C^3$ of length $N^3=\frac{N^1(N^2+2)}{k}$.
\end{constr}

These constructions result in the following lower bounds for $K(n,k), k\ge 3$.  
\begin{table}[H]
\caption{Lower bounds for $K(n,k)$}
\label{LowerBounds}
\centering
\begin{tabular}{l l l}
$K(n,2)\ge \frac{77}{256}2^n$ & & (See \cite{AK})\\
$K(n,3)\ge 32\cdot3^{(n-8)/3}$ & for $n\ge 6$ & (See \cite{Singleton}) \\
$K(n,k)\ge (k+1)2^{\lfloor 2n/(k+1) \rfloor -1}$ & for $k$ odd and $\lfloor \frac{2n}{k+1} \rfloor \ge 2$ & (See \cite{Singleton})\\
$K(n,4) \succ \delta^n$ & for $0 < \delta <3^{1/3}$ & (See \cite{Klee})\\
$K(n,k) \succ \delta^n$ & for $k$ even and $0<\delta < 4^{1/k}$ & (See \cite{Klee})\\
$K(n,k)\gtrsim 4^{n/(k +1)}$ & for odd $k>3$ & (See \cite{Klee})\\
\end{tabular}  
\end{table}  

The last three inequalities in Table \ref{LowerBounds} are asymptotic bounds, where $f(n)\lesssim g(n)$ means \\${\liminf_{n\rightarrow \infty} g(n)/f(n) >0}$, and $f(n) \prec g(n)$ means ${\lim_{n\rightarrow \infty} g(n)/f(n) = \infty}$.

In addition to the previous constructions, the ``necklace'' construction of Paterson and Tuliani has been particularly important, leading to many new records for $K(n,k)$ \cite{Paterson}.  However, identifying arrangements of necklaces satisfying the conditions of that construction required a backtrack search, limiting the dimensions examined to $n\le 17$. The conditions placed upon the arrangement of necklaces also become more restrictive as $k$ increases, and for the range of dimensions $n$ examined, no suitable arrangements for codes of spread $k\ge 7$ were found \cite{Paterson}.

For $n\le 17$ and $k\le 7$ many of the current records for $K(n,k)$ (reported in Table \ref{EmpiricalLowerBounds}) have been set by computational methods, e.g. exhaustive search \cite{Kochut, Hood}, pruning based approaches \cite{Tuohy}, genetic algorithms \cite{PotterRobinson, Casella, DiazGomez, Kinny}, or other computational approaches \cite{Chebiryak, Wynn, Allison}.

\section{Generating an $(n+r,k+1)$ Circuit Code from an $(n,k)$ Circuit Code}
\subsection{Transition Sequences}
Every vertex of $I(n)$ corresponds to a binary vector of length $n$, so for each circuit $C=(x_1,\ldots,x_N)$ of $I(n)$ we can define a transition sequence $T=(\tau_1,\ldots,\tau_N)$ where $\tau_i$ denotes the position in which $x_i$ and $x_{i+1}$ (or $x_N$ and $x_1$) differ.  Using the convention that $x_1=\vec{0}$ for any circuit, we see that the transition sequence corresponds uniquely to the edges in $C$.  Furthermore, for any two vertices $x,x'$ of an $(n,k)$ circuit code $C$ there are exactly two transition sequences between $x$ and $x'$, corresponding to the two paths between $x$ and $x'$ in $C$, and these sequences are complements in $T$\footnote{Without the convention $x_1=\vec{0}$ it is possible to have $\ge3$ sequences of transition elements which could lead from $x$ to $x'$.  But these cannot simultaneously correspond to valid paths between $x$ and $x'$ in the same circuit code $C$.}  A useful result to which we shall refer is the following.

\begin{lem}[Singleton \cite{Singleton}]
\label{Fact1}
Let $C$ be a circuit code of spread $k$ and length $N\ge 2(k+1)$ with corresponding transition sequence $T$.  Then any $k+1$ cyclically consecutive elements of $T$ are all distinct.
\end{lem}

\subsection{A New Circuit Code Construction}
Now we describe the intuition behind the construction used to prove Theorem \ref{Thm1}.  Let $\{t_1,\ldots,t_n\}$ be the set of transition elements present in $T$, the idea is to strategically insert members of a new set of transition elements $\{s_1,\ldots,s_r\}$ (corresponding to adding dimensions to the hypercube) into $T$ so that the resulting transition sequence $T'$ forms a circuit code of spread $k+1$ in dimension $n+r$.  (This is not as simple as inserting $s_1$ after every segment of $k+1$ transitions in $T$, since that does not guarantee that any two vertices with distance $k$ in $I(n+r)$ are separated by no more than $k$ edges (transitions) in the new code, see the Appendix for one such counterexample.)

\begin{constr}
\label{Construction7}
Let $C$ be an $(n,k)$ circuit code of length $N\ge 2(k+1)$ with transition sequence $T=(\tau_1,\ldots,\tau_N)$.  Split $T$ into $T^1=(\tau_1,\ldots,\tau_{N/2})$ and $T^2=(\tau_{N/2+1},\ldots,\tau_N)$.  For $i=1,2$ further divide $T^i$ into $q=\lceil \frac{N}{2(k+1)}\rceil$ segments $T^{i,1},\ldots,T^{i,q}$ where $T^{i,1}$ consists of the first $k+1$ elements of $T^i$, $T^{i,2}$ consists of the next $k+1$ elements, etc.  Only segment $T^{i,q}$ may have $< k+1$ elements.

For $l=1,\ldots,r-1 (=\lceil \log_2 q \rceil)$ insert $s_l$ at the end of segment $T^{i,j}$ if $2^{l-1}$ is the largest power of 2 that divides $j$ , for $j=1,\ldots, q-1$\footnote{In the degenerate case $q=1$, insert $s_r=s_1$ at the end of segment $T^{i,1}$}.  At the end of segment $T^{i,q}$ insert $s_r$.  This yields a new series of segments $T^{\prime i,1},\ldots, T^{\prime i,q}$ where $|T^{\prime i,j}|=|T^{i,j}|+1$.  Combine these transition sequences $T^{\prime 1}=(\tau'_1,\ldots, \tau'_{N/2+q})$, $T^{\prime 2}=(\tau'_{N/2+q+1},\ldots, \tau'_{N+2q})$ into a sequence $T'=(T^{\prime 1},T^{\prime 2})$.  Then $T'$ is the transition sequence of an $(n+r,k+1)$ circuit code $C'$ with length $N'=N+2\lceil\frac{N}{2(k+1)}\rceil$.
\end{constr}

E.g. if $q=10$, this process yields:
% insert table here.
\begin{table}[H]
\centering
\begin{tabular} {|c| c c c c c|}
\hline
j&1&2&3&4&5\\
\hline
$T^{\prime i,j}$ & $(T^{i,1},s_1)$ & $(T^{i,2},s_2)$ & $(T^{i,3},s_1)$ & $(T^{i,4},s_3)$ & $(T^{i,5},s_1)$ \\
\hline
j & 6 & 7 & 8 & 9 & 10\\
\hline
$T^{\prime i,j}$ & $(T^{i,6},s_2)$ & $(T^{i,7},s_1)$ & $(T^{i,8},s_4)$ & $(T^{i,9},s_1)$ &$(T^{i,10},s_5)$\\
\hline
\end{tabular}
\end{table}

If $x$ is a vertex of $I(n)$ and $\tilde{n}<n$, we denote by $x^*$ the ``natural'' projection of $x$ onto $I(\tilde{n})$ formed by taking the first $\tilde{n}$ elements of the binary vector $x$.  There is an important relationship between the transition sequence $T'$ from Construction \ref{Construction7} and the transition sequence $T$ of the underlying $(n,k)$ circuit code $C$.

\begin{lem}
\label{ShortSubsequence}
Let $C$ be an $(n,k)$ circuit code satisfying the assumptions of Construction \ref{Construction7} with transition sequence $T$, and let $C'$ be the resulting $(n+r,k+1)$ circuit code with transition sequence $T'$.  Let $x,x'\in C'$ and let $\hat{T}$ be a shortest transition sequence in $T'$ from $x$ to $x'$.  Then $\hat{T}\cap\{t_1,\ldots,t_n\}$ is a shortest transition sequence in $T$ between $x^*$ and $x^{\prime *}\in C$.
\end{lem}
\begin{proof}
Let $x,x'\in C'$ and let $\hat{T}$ be a transition sequence between them in $T'$, then $\hat{T}^C$ (the complement of $\hat{T}$ in $T'$) is also a transition sequence from $x$ to $x'$.  Furthermore, $x^*$ and $x^{\prime *}\in C$.  It is necessary that the subsequence $\hat{T}\cap \{t_1,\ldots,t_n\}$ (a segment of $T$) is a transition sequence in $T$ from $x^*$ to $x^{\prime *}$.  Since there are only two such sequences we conclude that $\hat{T}$ contains as a subsequence one of the transition sequences in $T$ from $x^*$ to $x^{\prime *}$ and $\hat{T}^C$ contains as a subsequence the other one.

Now suppose $|\hat{T}|\le|\hat{T}^C|$, then $|\hat{T}|\le N+q$ and $\hat{T}$ contains no transitions spaced $\ge N+q$ apart in $T'$.  For any $\tau'_i, \tau'_j \in T'$ spaced $N+q$ apart, $\tau'_i\in \{t_1,\ldots,t_n\} \iff \tau'_j\in \{t_1,\ldots,t_n\}$.  Thus $|\hat{T}|\le|\hat{T}^C|$ implies $|\hat{T}\cap \{t_1,\ldots,t_n\}|\le |\hat{T}^C\cap \{t_1,\ldots,t_n\}|$.  Hence the shorter transition sequence between $x$ and $x'$ (in $T'$) also contains the shorter transition sequence between $x^*$ and $x^{\prime *}$ (in $T$).
\end{proof}

Figure \ref{Fig1} illustrates this, showing a $(3,2)$ circuit code $C$ with transition sequence $T=(2,1,3,2,1,3)$ (on the left) and the $(4,3)$ circuit code $C'$ (on the right) with transition sequence $T'=(2,1,3,4,2,1,3,4)$ resulting from Construction \ref{Construction7}.  E.g. for $x=1100$ and $x'=1011$ the shortest path in $C'$ between $x$ and $x'$, indicated by dashed lines, ``contains as a subpath'' the shortest path in $C$ between $x^*=110$ and $x^{\prime *}=101$.

%insert figure here
 \begin{figure}[H]
  \caption{A $(3,2)$ Circuit Code and a $(4,3)$ Circuit Code}
  \label{Fig1}
 \centering
 \scalebox{0.6}
 {
 \begin{tikzpicture}[scale=1.25]
	 \tikzstyle{vertex}=[circle,minimum size=20pt,inner sep=0pt]
	 \tikzstyle{heavyedge} = [draw,line width=2.25pt,-]
	 \tikzstyle{mediumedge}=[draw,line width = 2.5pt,dotted,black]
	 \tikzstyle{edge} = [draw,-,black]
	 
	 \node[vertex] (v0) at (0,0) {$000$};
	 \node[vertex] (v1) at (0,1) {$010$};
	 \node[vertex] (v2) at (1,0) {$100$};
	 \node[vertex] (v3) at (1,1) {$110$};
	 \node[vertex] (v4) at (-1,-1) {$001$};
	 \node[vertex] (v5) at (-1,2) {$011$};
	 \node[vertex] (v6) at (2,-1) {$101$};
	 \node[vertex] (v7) at (2,2) {$111$};
	 
	 \draw[edge] (v0)--(v1)--(v3)--(v2)--(v0);
	 \draw[edge] (v4)--(v5)--(v7)--(v6)--(v4);
	 \draw[edge] (v0)--(v1)--(v5)--(v4)--(v0);
	 \draw[edge] (v2)--(v3)--(v7)--(v6)--(v2);

	 %4 dimensional
	 \node[vertex] (u0) at (5,0) {$0000$};
	 \node[vertex] (u1) at (5,1) {$0100$};
	 \node[vertex] (u2) at (6,0) {$1000$};
	 \node[vertex] (u3) at (6,1) {$1100$};
	 \node[vertex] (u4) at (4,-1) {$0010$};
	 \node[vertex] (u5) at (4,2) {$0110$};
	 \node[vertex] (u6) at (7,-1) {$1010$};
	 \node[vertex] (u7) at (7,2) {$1110$};
	 
	 \draw[edge] (u0)--(u1)--(u3)--(u2)--(u0);
	 \draw[edge] (u4)--(u5)--(u7)--(u6)--(u4);
	 \draw[edge] (u0)--(u1)--(u5)--(u4)--(u0);
	 \draw[edge] (u2)--(u3)--(u7)--(u6)--(u2);
	 
	 \node[vertex] (t0) at (8.5,-.67) {$0001$};
	 \node[vertex] (t1) at (8.5,.33) {$0101$};
	 \node[vertex] (t2) at (9.5,-.67) {$1001$};
	 \node[vertex] (t3) at (9.5,.33) {$1101$};
	 \node[vertex] (t4) at (7.5,-1.67) {$0011$};
	 \node[vertex] (t5) at (7.5,1.33) {$0111$};
	 \node[vertex] (t6) at (10.5,-1.67) {$1011$};
	 \node[vertex](t7) at (10.5,1.33) {$1111$};
	 
	 \draw[edge] (t0)--(t1)--(t3)--(t2)--(t0);
	 \draw[edge] (t4)--(t5)--(t7)--(t6)--(t4);
	 \draw[edge] (t0)--(t1)--(t5)--(t4)--(t0);
	 \draw[edge] (t2)--(t3)--(t7)--(t6)--(t2);
	 
	%draw edges linking the two 3-cubes
	\draw[edge](u0)--(t0);
	\draw[edge](u1)--(t1);
	\draw[edge](u2)--(t2);
	\draw[edge](u3)--(t3);
	\draw[edge](u4)--(t4);
	\draw[edge](u5)--(t5);
	\draw[edge](u6)--(t6);
	\draw[edge](u7)--(t7);
	
	% draw circuits
	\draw[heavyedge](v0)--(v1)--(v3);
	\draw[heavyedge](v6)--(v4)--(v0);
	\draw[mediumedge](v3)--(v7)--(v6);
	
	\draw[heavyedge](u0)--(u1)--(u3);
	\draw[mediumedge](u3)--(u7)--(t7)--(t6);
	\draw[heavyedge](t6)--(t4)--(t0)--(u0);

 \end{tikzpicture}
 }
 \end{figure}
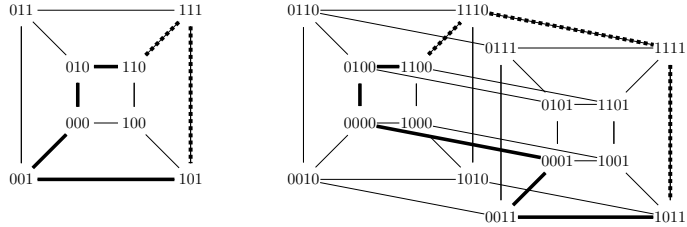

We now have everything we need to proceed to the main proof.

\begin{proof}[Proof of Theorem \ref{Thm1}]
Let $C$ be an $(n,k)$ circuit code with length $N\ge 2(k+1)$ and transition sequence $T$.  Apply Construction \ref{Construction7} to get a new transition sequence $T'$.  Clearly $T'$ forms a circuit $C'$ as each transition element appears an even number of times.  So we only need to show that $C'$ has spread $k+1$, and by Lemma \ref{Fact2} it thus suffices to show for all vertices $x, x'\in C'$ that $d_{C'}(x,x')\ge k+1 \Rightarrow d_{I(n+r)}(x,x')\ge k+1$.

Suppose that $x$ and $x'$ are vertices of $C'$ with $d_{C'}(x,x')\ge k+1$.  Let $\hat{T}$ denote the segment of $T'$ that is the transition sequence between $x$ and $x'$ (the shorter segment).  (So $\hat{T}$ may ``start'' in $T^{\prime 1}$ and end in $T^{\prime 2}$, or the reverse, or may be entirely contained in $T^{\prime i}$.)  Finally, let $A=\hat{T}\cap \{t_1,\ldots, t_n\}$ and $B=\hat{T} \cap \{s_1,\ldots, s_r\}$, thus $d_{C'}(x,x')=|A|+|B|$.

If $|B|=0$ then $|A|=k+1$.  In this case, by Lemma \ref{Fact1}, $d_{I(n)}(x^*,x^{\prime *})=k+1$, so $d_{I(n+r)}(x,x')=k+1$.  So we may assume $|B|>0$.

Let $\tau'_{\alpha}$ be the first element of $\hat{T}$ in $B$ and let $\tau'_{\beta}$ be the last element of $\hat{T}$ in $B$\footnote{Here ``first'' and ``last'' are with respect to ordering in $\hat{T}$, not $T'$.  So it is possible that $\alpha > \beta$.}.  If $|B|=1$ then an element of $\{s_1,\ldots,s_r\}$ appears an odd number of times in $\hat{T}$.  If $\tau'_{\alpha}$ and $\tau'_{\beta}$ are in the same segment $T^{\prime i}$ (i.e. $1\le \alpha , \beta \le N/2+q$ or $N/2+q+1 \le \alpha , \beta \le N+2q$) then let $s_p\in B$ have the maximum index (from $1$ to $r$) of all elements in $B$.  Observe that $s_p$ must occur exactly once in $\hat{T}$, otherwise (by construction) $s_w$ appears in $\hat{T}$ between two appearances of $s_p$ for some $w>p$, but this contradicts the definition of $s_p$.  If $\tau'_{\alpha}\in T^{\prime 1}$ and $\tau'_{\beta}\in T^{\prime 2}$ then $s_r\in \hat{T}$.  Since $s_r$ occurs exactly twice in $T'$ and in positions spread $N/2+q$ ($=|T'|/2$) apart, $s_r$ can only occur once in $\hat{T}$ (else the other sequence connecting $x$ and $x'$ would be shorter).  A similar analysis holds if $\tau'_{\alpha}\in T^{\prime 2}$ and $\tau'_{\beta}\in T^{\prime 1}$, and thus in all cases some $s_l \in \hat{T}$ occurs an odd number of times.

Now $d_{I(n+r)}(x,x')=d_{I(n)}(x^*,x^{\prime *})$+\# of members of $\{s_1,\ldots,s_r\}$ occuring an odd \# of times in $\hat{T}$.  If $d_{I(n)}(x^*,x^{\prime *})\ge k$ this is $\ge k+1$.  If $d_{I(n)}(x^*,x^{\prime *})<k$ then by Lemma \ref{ShortSubsequence} and since $C$ has spread $k$, $|A|=d_C(x^*,x^{\prime *})=d_{I(n)}(x^*,x^{\prime *})$.  Furthermore, each element of $A$ and $B$ appears exactly once (by Lemma \ref{Fact1}, the observation that $|A|\le k \Rightarrow |B|\le 2$, and the fact that consecutive elements of $B$ differ when $|\hat{T}|\le N+q$).  Thus $d_{I(n+r)}(x,x')=|A|+|B|=d_{C'}(x,x')\ge k+1$.
\end{proof}

\section{A New Lower Bound for $K(n,4)$}
Singleton \cite{Singleton} remarks that for $k\ge 4$ and even, the best lower bound available for $K(n,k)$ seems to be applying the third lower bound given in Table \ref{LowerBounds} to $K(n,k+1)$ (as every circuit code of spread $k+1$ is also a circuit code of spread $k$).  In particular, for $k=4$ this gives $K(n,4)\ge 6\cdot2^{\lfloor 2n/6\rfloor -1}$.  Construction \ref{Construction6} leads to a stronger asymptotic bound, $K(n,4)\succ \delta^n$ for $0<\delta<3^{1/3}$.  We will now prove that Theorem \ref{Thm1} gives a non-asymptotic lower bound that is stronger than $6\cdot2^{\lfloor 2n/6\rfloor -1}$ for $n\ge 86$.

First we establish the following claim, our argument is a minor modification of the one given in Chapter 17 of \cite{Grunbaum}.
\begin{lem}
\label{Claim1}
For $n\ge 6$ there exists an $(n,3)$ circuit code $C$ with length $N$ satisfying\\ ${32\cdot3^{(n-8)/3}\le N \le \frac{24}{22}32\cdot3^{(n-8)/3}}$.
\end{lem}
\begin{proof}
Let $C$ be an $(n,3)$ circuit code with transition sequence $T$.  Suppose that $t_i$ occurs $m$ times in $T$.  Construction $S5$ of \cite{Grunbaum} states that there is an $(n+3,3)$ circuit code $C'$ with length $N'=N+8m$, and $t_i$ occurs $3m$ times in the new transition sequence $T'$.  Note that if $N$ is divisible by 4 and $t_i$ appears $\frac{N}{4}$ times in $T$, then $N'=3N$ and $t_i$ appears $3m=\frac{N'}{4}$ times in $T'$.

For $n=6,7,8$ consider the following transition sequences for $(n,3)$ circuit codes.  Note that $|T_6|=16$, $|T_7|=24$, and $|T_8|=32$.  Also, 5 occurs 4 times in $T_6$, 2 occurs 6 times in $T_7$, and 8 occurs 8 times in $T_8$.

\begin{table}[H]
\centering
\begin{tabular}{l}
$T_6=(1,5,2,6,3,5,4,6,1,5,2,6,3,5,4,6)$\\
$T_7=(5,2,6,1,7,2,5,3,6,2,7,4,5,2,6,1,7,2,5,3,6,2,7,4)$\\
$T_8=(5,2,6,8,1,7,2,8,5,3,6,8,2,7,4,8,5,2,6,8,1,7,2,8,5,3,6,8,2,7,4,8)$\\
\end{tabular}
\end{table}

Therefore by Construction S5 we see that for any $p\in \NN$, in dimension $n=6+3p$ there exists an $(n,3)$ code with length $N=\frac{16}{15}32\cdot3^{(n-8)/3}$, in dimension $n=7+3p$ there exists an $(n,3)$ code with length $N=\frac{24}{22}32\cdot3^{(n-8)/3}$, and in dimension $n=8+3p$ there exists an $(n,3)$ code with length $N=32\cdot3^{(n-8)/3}$.
\end{proof}

\begin{proof}[Proof of Theorem \ref{Thm2}]
Theorem \ref{Thm1} implies $K(n+r,4)\ge N+2\lceil \frac{N}{2\cdot4}\rceil\ge \frac{5}{4}N$, where $N\ge 2\cdot4$ is the length of an $(n,3)$ circuit code and $r=\lceil \log_2 \frac{N}{2\cdot4}\rceil+1$.  From Lemma \ref{Claim1} we know that for $n\ge 6$ there exists a circuit code $C$ of length $N$ and $32\cdot3^{(n-8)/3}\le N \le \frac{24}{22} 32\cdot3^{(n-8)/3}$.  Using this code we have $K(n+r,4)\ge 40\cdot3^{(n-8)/3}$ and $r\le\lfloor \log_2 \frac{N}{2\cdot4}\rfloor+2$.

Now $2^{.53}>3^{1/3}$ so $r\le 2+\lfloor \log_2 \frac{24}{22}4\cdot3^{-8/3}\cdot2^{.53n}\rfloor\le .53n$.  Hence $K(\lfloor 1.53n\rfloor,4)\ge 40\cdot3^{(n-8)/3}$ for $n\ge6$.  And making the change of variables $u=1.53n$ we get $K(\lfloor u\rfloor,4)\ge 40\cdot3^{(\lfloor .6535u\rfloor-8)/3}$.
\end{proof}

A simple analysis shows that the lower bound of Theorem \ref{Thm2} exceeds $6\cdot2^{\lfloor 2n/6\rfloor -1}$ for $n\ge 86$.

\section{Computational Results}
\subsection{Methodology}
The efficacy of Construction \ref{Construction7} was tested by applying it to circuit codes of spreads 2-9 in dimensions 3-30\footnote{Circuit code construction and testing code is available from the author upon request.}.  Table \ref{EmpiricalLowerBounds} lists the greatest lower bound found for each $(n,k)$ combination.  The table was constructed as follows.  For spreads 2-7 and dimensions 3-30 we seeded the table with empirical results from \cite{Singleton, Deimer, Hood, Palombo, Allison} which collectively survey all empirical records of which we are aware, for spreads 8 and higher we seeded the table by using the exact bounds of Table \ref{ExactBounds} and the non-asymptotic lower bounds of Table \ref{LowerBounds}.  

Next, we applied Constructions \ref{Construction1} - \ref{Construction4} (collectively the ``Singleton'' constructions), the construction of Deimer (Construction \ref{Construction5}), and the construction of Klee (Construction \ref{Construction6}).  Because these constructions were applied sequentially we iterated applying the constructions until there was no improvement in any entry of the table.  To this ``initial'' table we then applied Construction \ref{Construction7} to the column corresponding to codes of spread $k$, replacing the appropriate entry in column $k+1$ of the table if a larger lower bound was found.  Each time after applying Construction \ref{Construction7} to codes of spread $k$ we repeated the iterative application of the constructions of Singleton, Deimer, and Klee to propagate any further improvements in the lower bounds before applying the construction to codes of spread $k+1$.  Finally, after applying the construction to codes of all spreads we iteratively applied the constructions from Singleton, Deimer, and Klee once more.

Construction \ref{Construction6} was applied to our table as follows.  Let $C$ be an $(n,k)$ circuit code with length $N>2(k+1)^2$, and let $T=(\tau_1,\ldots,\tau_N)$ be its transition sequence.  Split $T$ into $T^1=(\tau_1,\ldots,\tau_{N/2})$, $T^2=(\tau_{N/2+1},\ldots,\tau_N)$ and subdivide $T^i$ into $q=\lceil \frac{N}{2(k+1)}\rceil$ segments $T^{i,1},\ldots,T^{i,q}$ of length $\le k+1$ as in Construction \ref{Construction7} (where only segment $T^{i,q}$ may have length $<k+1$).  Note that $q>k+1$.  For $i=1,2$ create a new transition sequence $T^{\prime i}$ by inserting the new transition element $t_{n+1}$ at the end of the first $p=(k+1)\lceil\frac{N}{2(k+1)}\rceil-\frac{N}{2}$ segments $T^{\prime i,j}$ of $T^{\prime i}$.  Finally combine $T^{\prime 1}, T^{\prime 2}$ into $T'=(T^{\prime 1},T^{\prime 2})$.  Observe that $t_{n+1}$ occurs an even number of times in $T'$ and any two occurences of $t_{n+1}$ are separated by $\ge k+1$ other transition elements.  The resulting circuit code $C'$ has dimension $n+1$ and spread $k$ (but not necessarily spread $k+1$) and length $N'=2(k+1)\lceil\frac{N}{2(k+1)}\rceil$, satisfying the divisibility criterion of Construction \ref{Construction6}.  Because this method does not generate all $(n+1,k)$ circuit codes with length divisible by $k+1$, we also indicate in Table \ref{EmpiricalLowerBounds} when an entry exceeds the asymptotic lower bounds from Table \ref{LowerBounds} which are derived from Construction \ref{Construction6}.

\begin{table}
\centering
\caption{Lower Bounds for $K(n,k)$ (Previous Best Bound in Parentheses)}
\label{EmpiricalLowerBounds}
%\begin{tabular}
\begin{filecontents*}{techReportRecords.csv}
n/k,2,3,4,5,6,7,8,9
3,6c,6c,6c,6c,6c,6c,6c,6c
4,8c,8c,8c,8c,8c,8c,8c,8c
5,14c,10c,10c,10c,10c,10c,10c,10c
6,26c,16c,12c,12c,12c,12c,12c,12c
7,48c,24c,14c,14c,14c,14c,14c,14c
8,96c,36c,22c,16c,16c,16c,16c,16c
9,188,64,30c,24c,18c,18c,18c,18c
10,362,102,46c,28c,20c,20c,20c,20c
11,668,160,70,40c,30c,22c,22c,22c
12,1340,288,102,60,36c,32c,24c,24c
13,2584,494,182,80,50c,36c,26c,26c
14,4934,812,280,106,68,48c,38c,28c
15,9868,1380,480,210,88,60,42,40c
16,19740,2240,768,288,118,76,46,44c
17,39840,3910,1224,476,204,102,54,48
18,78848,5212,1530,570,238,116,68(60)ab,52
19,157696,7818,2040,712,284,134,78,60
20,315392,10424,2688,950,330,152,86,80
21,630784,15634,3400,1140,436,198,116(98)ab,88
22,1261568,20848,4488,1422,510,234(228)ab,132(114)ab,100
23,2523136,31266,5910,1898,608,266(262)b,148(128)b,110
24,5046272,41696,7480,2280,714,310(304)b,168(158)b,124
25,10092544,62530,9870,2846,932,390,188(176)b,160
26,20185088,83392,13248,3794,1086,466(452)b,236(202)ab,176
27,40370176,125058,20304,4560,1304,532(518)b,272(234)ab,200
28,80740352,166784,34704,5690,1530,618(608)b,308(268)b,222
29,161480704,250114,57246,7586,1996,774,348(328)b,248
30,322961408,333568,97846,9120,2328,930(900)b,396(368)b,320
\end{filecontents*}
\csvautotabular{techReportRecords.csv}
%\end{tabular}
\caption*{
a = previous record was also exceeded ``directly'' by applying new construction\\
b = record exceeds Klee's asymptotic lower bound\\
c = value known to be optimal}
\end{table}

\subsection{Discussion of Computational Results}
Our construction found several new circuit codes for spreads of 7 and 8.  Because codes of spreads 2-7 and dimensions 3-30 have been well-studied (see \cite{Hood, Palombo} for surveys) the improvements noted in Table \ref{EmpiricalLowerBounds} for codes of spread 7 are perhaps the most significant.  The new lower bound for a $(22,7)$ code results from applying Construction \ref{Construction7} to the $(17,6)$ code with length 204 found by Paterson and Tuliani \cite{Paterson}, and the full code is given in the Appendix.  The new lower bounds for codes of spread 7 and dimension $>$ 22 come about by applying the Singleton and Deimer constructions to the $(22,7)$ code.

The chief advantage of our construction is that it is very easy to implement, allowing the better studied codes of smaller spreads to be leveraged to generate codes of larger spreads, where the spread is too large for computer search. This adds another construction (in addition to Constructions \ref{Construction1} - \ref{Construction6}) to generate non-trivial codes for large spreads.  As the results for spreads $k=7,8$ indicate, the construction is additive to Constructions \ref{Construction1}-\ref{Construction6}.  However the results for spread $k+1=9$ indicate that the success of this approach relies on good starting codes for spread $k$.

\section{Conclusions}
In this note we presented a simple method for constructing a circuit code of spread $k+1$ from a circuit code of spread $k$.  This construction leads to record code lengths for circuit codes of spread $k=7,8$ and dimension $22\le n\le 30$.  We also derived a new lower bound on the length of circuit codes of spread 4, which improves upon the current bound for $n\ge 86$.

Some of the records in Table \ref{EmpiricalLowerBounds} stood for at least 32 years before being broken by the method described here, however we believe that further improvements of the lower bounds on $K(n,k)$ are still possible.  In  particular, Construction 5 from \cite{Singleton} describes how to extend an $(n,7)$ circuit code under certain conditions on how close a specific pair of transition elements appear in the transition sequence.  While applying that construction directly does not improve the lower bounds in the table (we tried!) the transition sequences arising from combining Construction \ref{Construction7} with the construction method of \cite{Paterson} are highly structured, suggesting that a modification of that approach may succeed.
\\
\\
\textbf{Acknowledgements:}
The author thanks Stephen Chestnut and Eric Harley for generously reviewing earlier versions of this paper, and for many helpful suggestions which greatly improved the final version.

\bibliographystyle{unsrt}
\bibliography{CircuitCodes}

\appendix
\section{An $(n,k)$ Circuit Code that Cannot be Trivially Extended to an $(n+1,k+1)$ Circuit Code}
The following transition sequence from \cite{Klee} results in a $(6,2)$ circuit code of length 24:\\
$T=(1, 2, 6, 4, 5, 6, 1, 3, 5, 4, 6, 5, 1, 2, 6, 4, 5, 6, 1, 3, 5, 4, 6, 5)$.  

This code cannot be extended to a $(7,3)$ circuit code by inserting the new transition element 7 after the end of every segment of $T$ of length 3.  There are 3 potential new transition sequences $T'$, starting with: $(7,1,2,6)$, or $(1,7,2,6)$, or $(1,2,7,6)$ and inserting 7 after the end of every segment in $T$ of length 3 after the initial appearance of 7.

The reader can verify that none of the 3 potential transition sequences arising in this way has spread 3.

\section{Transition Sequence for a $(22,7,234)$ Circuit Code}
Table \ref{TransitionElements} lists the transition elements for the (22,7) circuit code of length 234 generated using the construction of Theorem \ref{Thm1}.  This is arguably the most important code we discovered as all other new $(n,7)$ codes and all but one new $(n,8)$ code are built using this code in conjunction with other constructions.  Unlike the notation in \cite{Paterson} we consider the positions of the vectors in $I(n)$ (and thus the range for transition elements) to be $1\text{ to }n$ not $0\text{ to }(n-1)$.  Transition sequences for other codes are available from the authors upon request.

\begin{table}[H]
\centering
\caption{Transition Sequence for $(22,7,234)$ Circuit Code}
\label{TransitionElements}
\begin{tabular}{l l}
\hline\\
$(n,k, N)$ & Transition Elements (read row-wise)\\
\hline \hline\\
(22,7,234) & (6	12	4	3	16	8	9	18	17	13	4	12	15	16	5	19	14	13	9	1	2	10	6	18	14	5	8	9	15	7\\
& 6	20	2	11	12	3	16	7	15	18	1	2	8	17	16	12	4	19	5	13	9	17	8	11	12	18	1	10	9	5\\
& 14	15	6	21	2	10	1	4	5	11	3	18	2	15	7	8	16	12	3	19	11	14	15	4	13	12	8	18	17	1\\
& 9	5	13	4	7	20	8	14	6	5	1	10	11	18	2	15	6	14	17	1	7	19	16	15	11	3	22	4	12	8\\
& 16	7	10	11	18	17	9	8	4	13	14	5	19	1	9	17	3	4	10	2	18	1	14	6	7	15	11	2	20	10\\
& 13	14	3	12	11	7	18	16	17	8	4	12	3	6	19	7	13	5	4	17	9	10	18	1	14	5	13	16	17	6\\
& 21	15	14	10	2	3	11	7	18	15	6	9	10	16	8	7	19	3	12	13	4	17	8	16	18	2	3	9	1	17\\
& 13	5	20	6	14	10	1	9	12	13	18	2	11	10	6	15	16	7	19	3	11	2	5	22)\\
\end{tabular}
\end{table}

\end{document}